\documentclass[12pt]{article}%
   
\usepackage{amsmath,enumerate}
\usepackage{amsfonts}
\usepackage{amssymb}

\usepackage{color}

\usepackage{comment}

\newtheorem{theorem}{Theorem}[section]

\newtheorem{corollary}[theorem]{Corollary}

\newtheorem{lemma}[theorem]{Lemma}

\newtheorem{proposition}[theorem]{Proposition}

\newenvironment{proof}[1][Proof]{\noindent\textbf{#1.} }
{\hfill \ \rule{0.5em}{0.5em}}

\begin{document}

\title{$C_{2k}$-saturated graphs with no short odd cycles}
\author{Craig Timmons\thanks{Department of Mathematics and Statistics, California State University Sacramento.  
This work was supported by a grant from the Simons Foundation (\#359419).} }

\date{}
\maketitle

\begin{abstract}
The saturation number of a graph $F$, written $\textup{sat}(n,F)$,
is the minimum number of edges in an $n$-vertex $F$-saturated graph.
One of the earliest results 
on saturation numbers is due to 
Erd\H{o}s, Hajnal, and Moon
who determined $\textup{sat}(n,K_r)$ for all $r \geq 3$.
Since then, saturation numbers of various graphs and hypergraphs have been studied. 
Motivated by Alon and Shikhelman's generalized Tur\'{a}n function, 
Kritschgau et.\ al.\ defined $\textup{sat}(n,H,F)$ to
be the minimum number of copies of $H$ in an $n$-vertex
$F$-saturated graph.
They proved, among other things, that 
$\textup{sat}(n,C_3,C_{2k}) = 0$ for all $k \geq 3$ and 
$n \geq 2k +2$.  We extend this result to all odd cycles
by proving that for any odd integer $r \geq 5$, $\textup{sat}(n , C_r ,C_{2k} ) = 0$ 
 for all $2k \geq r+5$ and $n \geq 2kr$.    
\end{abstract}


\section{Introduction}

Let $F$ be a graph.  A graph $G$ is \emph{$F$-saturated} if $G$ does not contain $F$ as 
a subgraph and adding any missing edge to $G$ creates a copy of $F$.
Write $\textup{sat}(n, F)$ for the minimum number of edges in 
an $F$-saturated $n$-vertex graph.  The function $\textup{sat}(n, F)$ is the
\emph{saturation number of $F$}.  
Much research has been done on estimating saturation numbers of graphs. 
It is not too difficult to see that
 $\textup{sat}(n,K_3) = n-1$ where the upper bound follows by considering $K_{1,n-1}$.  
 The lower bound is true since any $K_3$-saturated graph must 
 be connected, and a connected $n$-vertex graph has at least $n-1$ edges.  
The case when $F$ is
 an arbitrary complete graph was solved by Erd\H{o}s, Hajnal, and Moon \cite{ehm}.
 
 \begin{theorem}[Erd\H{o}s, Hajnal, and Moon \cite{ehm}]\label{theorem ehm}
For $r \geq 3$,
\[
\textup{sat}(n , K_r) = (r-2) ( n - r +2) + \binom{r-2}{2}.
\]
Furthermore, if $G$ is an $n$-vertex $K_r$-saturated
graph with $\textup{sat}(n , K_r)$ edges, then $G$ is isomorphic to 
$K_{r-2} + \overline{K_{n-r +2}}$.  
\end{theorem}

The graph $K_{r-2} + \overline{K_{n - r+2 }}$ is the join of a complete
graph with $r-2$ vertices and an independent set with $n-r +2 $ vertices.  

Theorem \ref{theorem ehm} gives an exact result for the saturation 
number of any complete graph.  A theorem of 
K\'{a}szonyi and Tuza \cite{kt} 
gives an upper bound on the saturation number of any graph $F$,
and shows $\textup{sat}(n, F) $ is always at most linear in $n$.

\begin{theorem}[K\'{a}szonyi and Tuza \cite{kt}]\label{theorem kt}
For any graph $F$, there is a constant $C_F > 0$ such that
\begin{equation}\label{inequality kt}
\textup{sat}(n , F) < C_F n.
\end{equation}
\end{theorem}

Theorems \ref{theorem ehm} and \ref{theorem kt} are just two results
on graph saturation.  There are many others that concern specific graphs $F$ (cycles, stars, trees, etc.),
as well as saturation in different settings, such as random graphs and hypergraphs.  
The special case where $F$ is a cycle has been studied extensively  
\cite{barefoot, chen1, chen2, fk, gould, zls}.  
The excellent 
survey of Faudree, Faudree, and Schmitt \cite{sat survey} highlights many results and has
an extensive bibliography.

Motivated by a generalization of the Tur\'{a}n function introduced by 
Alon and Shikhelman \cite{as},
Kritschgau, Methuku, Tait and the author \cite{kmtt} defined the following
function which generalizes saturation numbers.
Let $H$ and $F$ be graphs.  Write
\[
\textup{sat}(n , H ,F)
\]
for the minimum number of copies of $H$ in an $n$-vertex $F$-saturated graph.
Note that $\textup{sat}(n,K_2 ,F) = \textup{sat}(n ,F)$ for any graph $F$.

A case that was studied in \cite{kmtt} is when both $H$ and $F$ are cycles.
Some combinations are trivial.  When $r , k \geq 1$, a complete bipartite graph
shows $\textup{sat}(n , C_{2r+1} , C_{2k+1} ) = 0$ for $n \geq 2k+1$.
Others are not nearly as easy.  For example, $\textup{sat}(n , C_3 ,C_4)$ 
which is the minimum number of triangles in a $C_4$-saturated graph.
It is not known if this function is 0 for infinitely many $n$,
or if it is positive for infinitely many $n$, a question first posed in \cite{kmtt}.

The case that is the focus of this paper is when $H$ is an odd cycle and $F$ is an even cycle.  
Kritschgau et.\ al.\ \cite{kmtt} proved that for any $k \geq 5$,
\[
\textup{sat}(n , C_3 , C_k ) = 0 
\]
provided $n \geq 2k +2$.  
Our main result extends this to all odd integers $r \geq 5$.

\begin{theorem}\label{theorem main}
Let $r \geq 5$ be an odd integer and let $2k \geq r +5$.
For any integer  $n \geq 2kr$,
\[
\textup{sat}(n , C_r , C_{2k}) = 0.
\]
\end{theorem}

As mentioned earlier, a complete bipartite graph shows $\textup{sat}(n,C_{2r+1} , C_{2k+1}) = 0$ for 
$n \geq 2k+1$.  This, the results of \cite{kmtt}, and Theorem \ref{theorem main} give the following Corollary.

\begin{corollary}\label{corollary}
For any odd integer $r \geq 3$ and any $k \geq \lceil \frac{r+5}{2} \rceil$, 
\[
\textup{sat}(n, C_r , C_k ) = 0
\]
for all $n \geq 2kr$. 
\end{corollary}

The lower bound on $n$ in Theorem \ref{theorem main} was chosen for simplicity.
Our proof shows that $n$ can be as small as $\frac{r+2}{3} (2k+1)$ in the case that $r \equiv 1 (\textup{mod}~3)$,
as small as $\frac{r+1}{3} (2k +1)$ in the case that $r \equiv 2 (\textup{mod}~3)$, and
finally as small as $\frac{2rk}{3} + \frac{r+6}{3}$ when $r \equiv 0 ( \textup{mod}~3)$.  
Also, the construction proving Theorem \ref{theorem main} has no 
odd cycle of length $r'$ for $3 \leq r' \leq r$ (see Sections \ref{showing cr free} and \ref{finish proof}).

Corollary \ref{corollary} shows that 
$\textup{sat}(n , C_r , C_k) = 0$ for all odd $r \geq 3$ where $k$ and $n$ are large enough
 in terms of $r$.  
The case when $r$ is even is not as well understood.  
In \cite{kmtt}, it is shown that 
$\textup{sat}(n , C_4 ,C_6) \leq \frac{n-1}{5}$ provided 
$n \equiv 1 (\textup{mod}~10)$.  
However, 
like $\textup{sat}(n,C_3,C_4)$,
we do not even know if $\textup{sat}(n ,C_4, C_6)$ is positive for infinitely many $n$. 
For further discussion, see \cite{kmtt}, but in general,
we do not have a complete understanding of $\textup{sat}(n,C_{2r},C_{2k})$. 

The rest of this paper contains the proof of Theorem \ref{theorem main}.
We take a moment here to describe some of the ideas behind the proof.
Let $r \geq 5$ be an odd integer and 
suppose $k$ is an integer with $ 2k \geq r +5$.
The first step is to construct a small $C_r$-free $C_{2k}$-saturated graph on
$\approx \frac{2kr}{3}$ vertices.  Most of the paper is devoted to constructing 
this graph (Section \ref{grk}), and showing that it has the properties we need:
$C_r$-free (Section \ref{showing cr free}), and $C_{2k}$-saturated
(Sections \ref{showing c2k free} and \ref{showing saturation}).
This graph, which we will call $G_r(k)$, will have a set of vertices $B$,
of size $\approx \frac{k}{2}$, that forms an independent set.
Also, all vertices in $B$ are joined to the same $\approx \frac{k}{2}$ vertices in $G_r(k)$.
Using a lemma from \cite{kmtt}, 
for any $t$ we can add $t$ new vertices
$v_1,\dots,v_t$ where each $v_i$ has the same neighborhood as $v \in B$.    
This new graph, obtained by adding these $t$ vertices to $G_r(k)$, will 
still be $C_r$-free and $C_{2k}$-saturated (Section \ref{finish proof}).
Since this can be done for any integer $t \geq 0$, we 
get $\textup{sat}(n , C_r ,C_{2k} )=0$ provided $n \geq | V( G_r (k) ) |$.


\section{Defining the graph $G_r (k)$}

In this section we define the graph $G_r(k)$ and  
certain subgraphs of $G_r(k)$ that are particularly important.
These subgraphs will be denoted by $H_i(k)$.  Additionally, we 
define certain paths in $G_r(k)$ which we will call \emph{core paths}. 
The rough idea is that
given a pair of nonadjacent vertices $z_1$ and $z_2$, 
 we will use a subgraph $H_i(k)$ to extend 
a core path from $z_1$ to $z_2$ to a path of length $2k-1$ from $z_1$ to $z_2$.  
We mention this ahead of time
to alert the reader to note when these important terms are defined.   
  
  
  \subsection{The graph $G_r(k)$}\label{grk}
 
Let $r \geq 5$ be an odd integer and let $k $ be an integer with $2k \geq r+5$.  
The construction of $G_r (k)$ depends on the residue class of $r $ modulo 3.  
The simplest case is when $r \equiv 1 ( \textup{mod}~3)$ so we begin here.
  
 \smallskip
  
Suppose $r \geq 7$ is odd and $r \equiv  1 ( \textup{mod}~3)$.  
Let 
\[
C= a_1, x_2, 3, a_4, x_5, 6 , \cdots , r-1 , a_r , x_{r+1}, r+2 , a_1
\]
be a cycle of length $r+2$.  For each $a_i$, add a new vertex $b_i$ whose only neighbor is $a_i$.
For each $x_{i+1}$, add a new vertex $y_{i+1}$ whose only neighbor is $x_{i+1}$.  This
 same step will be done when $r \equiv 0 ( \textup{mod}~3)$ and 
 when $r \equiv 2 ( \textup{mod}~3)$, so we call this process \emph{appending leaf vertices}.
 At this point, the degree of any $a_i$ or $x_{i+1}$ is 3, the degree of any $b_i$ or $y_{i+1}$ is 1,
 and the degree of all vertices in $\{3,6, \dots , r-1 ,r +2 \}$ is 2.  We 
 replace each $a_i$, $b_i$, $x_{i+1}$, and $y_{i+1}$ with 
 disjoint sets $A_i(k)$, $B_i(k)$, $X_{i+1}(k)$, and $Y_{i+1}(k)$ where 
 \begin{enumerate}
 \item $|A_i(k)| = |B_i (k) | = \lceil \frac{k}{2} \rceil$, $|X_{i+1} (k) | = |Y_{i+1}(k) | = \lfloor \frac{k}{2} \rfloor$,
 \item each $A_i(k)$, $B_i(k)$, $X_{i+1} (k)$, and $Y_{i+1}(k)$ is an independent set,
 \item the pairs $(A_i (k) , B_i (k))$, $(A_i(k) , X_{i+1}(k))$, and $(X_{i+1}(k) , Y_{i+1}(k))$ induce
 complete bipartite graphs,
 \item each vertex of $A_i(k)$ is adjacent to $i-1$ (in the case that $i=1$, each vertex in $A_1 (k)$ is adjacent to 
 $r+2$),
 \item each vertex of $X_{i+1}(k)$ is adjacent to $i+2$,
 \item for all $i \neq j$, $A_i (k) \cap A_j (k) = \emptyset$,  $B_i (k) \cap B_j (k) = \emptyset$,
  $X_{i+1} (k) \cap X_{j+1} (k) = \emptyset$, and  $Y_{i+1} (k) \cap Y_{j+1} (k) = \emptyset$.
 \end{enumerate}
 This completes the construction of $G_r (k)$ when $r \equiv 1 ( \textup{mod}~3)$.    
 Roughly speaking, we start with the cycle $C$, append leaf vertices $b_i$ and $y_{i+1}$ to 
 each $a_i$ and $x_{i+1}$, respectively, and then blow up $a_i$, $b_i$, $x_{i+1}$, $y_{i+1}$, 
 into sets of size $\approx \frac{k}{2}$.

 \smallskip
 
 Now we move onto the case when $r \equiv 2 ( \textup{mod}~3)$.  Here the construction is very
 similar, except we will start with a slightly different cycle of length $r+2$.
 
 Suppose $r \geq 5$ and $r \equiv 2 ( \textup{mod}~3)$.  Let
 \[
C =  a_1 ,x_2 , 3 , a_4 , x_5 , 6 , \cdots , r-2, a_{r-1} , x_r, r+1, r +2 , a_1 
 \]
 be a cycle of length $r+2$.  For each $a_i$ and $x_{i+1}$, append leaf vertices 
 $b_i$ and $y_{i+1}$ to $a_i$ and $x_{i+1}$, respectively.  We then blow up each
 $a_i$, $b_i$, $x_{i+1}$, and $y_{i+1}$ into sets $A_i(k)$, $B_i(k)$,
 $X_{i+1}(k)$, and $Y_{i+1} (k)$ as before.  
 
 \smallskip
 
 Next we address the case when $r \equiv 0 ( \textup{mod}~3)$ and $r \geq 9$.  Again the construction is similar to
 the previous cases, except now we start with the cycle 
 \[
C =  a_1 , x_2 ,3 , a_4 , x_5 , 6 , \cdots      ,r-6, a_{r-5},x_{r-4},r-3,r-2,a_{r-1},x_r, r+1,r+2,a_1
 \]
 which has length $r+2$.  For each $a_i$ and $x_{i+1}$, append leaf vertices 
 $b_i$ and $y_{i+1}$ to $a_i$ and $x_{i+1}$, and blow up
 $a_i$, $b_i$, $x_{i+1}$, and $y_{i+1}$ into $A_i(k)$, $B_i(k)$,
 $X_{i+1}(k)$, and $Y_{i+1} (k)$ as before.  
 
 \smallskip
 
 This completes the construction of $G_r (k)$ in all cases.

 
 \subsection{The graph $H(k)$}\label{hk}
 
 The graph $G_r(k)$ contains several 
 copies of a graph $H(k)$ that we will define now.  This subgraph will be a crucial part
 of the proof when we go to construct paths of length $2k-1$ between nonadjacent 
 vertices.
 Let $k \geq 3$ be an integer.  Let $H (k)$ be the graph whose vertex set is the disjoint union of 
 four independent sets $A(k)$, $B(k)$, $X(k)$, and $Y(k)$ where 
 $|A(k)|= |B(k) | = \lceil \frac{k}{2} \rceil$ and 
 $|X(k) | = |Y(k)| = \lfloor \frac{k}{2} \rfloor$.
 Every vertex in $A(k)$ is adjacent to every vertex in $B(k) \cup X(k)$, and every vertex in
 $X(k)$ is joined to every vertex in $A(k) \cup Y(k)$.  
 We write $A(k) = \{a_1 ,a_2 ,\dots , a_{ \lceil k/2 \rceil } \}$,
 $B(k) = \{b_1 , b_2 , \dots , b_{\lceil k/2 \rceil } \}$, and use similar notation
 for $X(k)$ and $Y(k)$.  
 
 For any $i$, the subgraph of $G_r(k)$ induced by $A_i (k) \cup B_i (k) \cup X_{i+1} (k) \cup Y_{i+1}(k)$
 is isomorphic to $H(k)$.  Write
 \[
 H_i (k)
 \]
 for this copy of $H(k)$ in $G_r(k)$.  An important observation is that since $r \geq 5$, 
 every $G_r(k)$ contains at least two copies of $H(k)$, namely $H_1(k)$ and $H_4(k)$.

\begin{lemma}[Path lengths in $H(k)$]\label{lemma paths}
Let $k \geq 5 $ and let 
\begin{center}
$\alpha_1 \in \{3,5, \dots , 2k-3 \}$, $\alpha_2 \in \{3,5, \dots , 2k-1 \}$, 
\end{center}
and
\begin{center}
$\beta_1 \in \{2,4, \dots , 2k-4 \}$, $\beta_2 \in \{2,4, \dots , 2k-2 \}$.  
\end{center}
For any 
$a \neq a' \in A(k)$, $b \neq b' \in B(k)$, $x \neq x' \in X(k)$, and 
$y \neq y' \in Y(k)$, there is a path of length 

\smallskip
\noindent
(i) $\alpha_1$ from $a$ to $x$, from $b$ to $a$, and from $y$ to $x$,

\smallskip
\noindent
(ii) $\alpha_2$ from $b$ to $y$,

\smallskip
\noindent
(iii) $\beta_1$ from $a$ to $a'$, and from $x$ to $x'$,

\smallskip
\noindent
(iv) $\beta_2$ from $b$ to $b'$, from $y$ to $y'$, from $b$ to $x$, and from $y$ to $a$.
\end{lemma}  
\begin{proof}
We will assume that $k$ is odd as the case when $k$ in even is the same.  
Let $t = \frac{k-1}{2}$.  For each of the assertions (i)-(iv), we will just 
give 
the longest path joining the pair of vertices under consideration.  Given the longest path,
it is not difficult to see how to shorten the path by removing vertices.

\smallskip
\noindent
(i) Paths of length $2k-3$, all of which miss one vertex $y_t \in Y (k)$, and 
one vertex $b_{t+1} \in B(k)$:

\begin{center}
from $a$ to $x$ ~ : ~ $a_1 b_1 a_2 b_2 \cdots a_t b_t a_{t+1} x_1 y_1 x_2 y_2 \cdots x_{t-1} y_{t-1} x_t$,

from $b$ to $a$ ~ : ~ $b_1 a_1 x_1 y_1 x_2 y_2 \cdots x_{t-1} y_{t-1} x_t a_2 b_2 a_3 b_3 \cdots a_t b_t a_{t+1}$,

from $y$ to $x$ ~ : ~ $y_1 x_1 a_1 b_1 a_2 b_2 \cdots a_t b_t a_{t+1} x_2 y_2 x_3 y_3 \cdots x_{t-1} y_{t-1} x_t$.
\end{center}

\smallskip
\noindent
(ii) Path of length $2k-1$ from $b$ to $y$:
\[
b_1 a_1 b_2 a_2 \cdots b_{t+1} a_{t+1} x_1 y_1 x_2 y_2 \cdots x_t y_t.
\]

\smallskip
\noindent
(iii) Paths of length $2k-4$, the first of which misses two vertices in $B(t)$ and one vertex in $Y(t)$; the second
of which misses two vertices in $Y(t)$ and one vertex in $B(t)$:
\begin{center}
from $a$ to $a$ ~ : ~ $a_1 b_1 a_2 b_2 \cdots a_{t-1} b_{t-1} a_t x_1 y_1 x_2 y_2 \cdots x_{t-1} y_{t-1} x_t a_{t+1}$,

from $x$ to $x$ ~ : ~ $x_1 y_1 x_2 y_2 \cdots x_{t-2} y_{t-2} x_{t-1} a_1 b_1 a_2 b_2 \cdots a_t b_t a_{t+1} x_t$.
\end{center}

\smallskip
\noindent
(iv) For a path of length $2k-2$ from $b$ to $b'$, we can extend the first path
in (iii) by adding $b_t$ and $b_{t+1}$.  The same can be done with the 
second path in (iii) to get a length $2k-2$ path from $y_{t-1}$ to $y_t$.  
For a path of length $2k-2$ from $b$ to $x$ or from $y$ to $a$, we can
delete $y_t$ or $b_1$, respectively, from the path in (ii).

\end{proof}


\subsection{Core paths}\label{core paths}

In this subsection, we define certain paths in $G_r(k)$ that are important for 
constructing paths of length $2k-1$.
A \emph{core path} in $G$ is a path $P$ such that 
\begin{enumerate}
\item $V(P) \cap B_i (k ) = \emptyset$ and $V(P) \cap Y_{i+1}(k) = \emptyset$ for all $i$, and
\item $|V (P) \cap A_i(k) | \leq 1$ and $|V(P) \cap X_{i+1}(k) | \leq 1$ for all $i$.
\end{enumerate}
If $P$ is a core path in $G_r(k)$, then $P$ intersects any $H_i(k)$ in at most two vertices.
An example of a core path in $G_9 (k)$ is 
\[
a_1, 11, 10, x_9,a_8 , 7 , 6 , x_5
\]
where $a_1 \in A_1 (k)$, $x_9 \in X_9 (k)$, $a_8 \in A_8 (k)$, and $x_5 \in X_5 (k)$.
In fact, any core path in $G_9 (k)$ from $a_1$ to $x_5$ is of the form
\[
a_1 , x_2 , 3 , a_4 , x_5 ~~~ \mbox{or}~~~ a_1, 11, 10, x_9,a_8 , 7 , 6 , x_5.
\]
Any core path between two vertices in $G_r (k)$ is unique up to the choice of
vertices from each $A_i (k)$ and $X_{i+1} (k)$.  This is easy to see as there
is a natural graph homomorphism $f : V( G_r (k) ) \rightarrow V(G_r)$ where
$G_r$ is the graph obtained by collapsing each $A_i(k)$ to a single vertex $a_i$, 
and doing the same for $B_i(k)$, $X_{i+1}(k)$, and $Y_{i+1}(k)$.
Thus, $G_r$ is simply an $(r+2)$-cycle with leaf vertices appended to each $a_i$ and each $x_{i+1}$.
We formalize this discussion in the following proposition but first,
given a path $P$, write 
\[
l (P)
\]
 for the length of $P$.  If the first vertex of $P$ is $z_1$ and the last vertex is $z_2$, we 
 call $P$ a \emph{$z_1$,$z_2$-path}.
 If $P$ is a core path from $z_1$ to $z_2$, we say that $P$ is a \emph{$z_1$,$z_2$-core path}.

\begin{proposition}\label{proposition core paths}
Let $r \geq 5$ be odd and let $k$ be an integer with $2k \geq r+5$.
If $z_1$ and $z_2$ are two nonadjacent vertices in $G_r(k)$ such that $z_1 ,z_2 \notin B_i(k) \cup Y_{i+1} (k)$
for all $i$, then there are two edge disjoint $z_1$,$z_2$-core paths $P_1$ and $P_2$ 
where 
\[
l(P_1) + l(P_2)  = r+2.
\]
In particular, there is a $z_1$,$z_2$-path $P$ with $l(P)  \in \{3,5, \dots, r \}$, and
$P$ intersects any $H_i(k)$ in at most two vertices.  
\end{proposition}

One last definition is needed before moving onto the proof of Theorem \ref{theorem main}.
Let
\[
D_r = V(G_r (k) ) \backslash \left( \bigcup_i A_i(k) \cup B_i (k) \cup X_{i+1}(k) \cup Y_{i+1}(k) \right).
\]
For instance, in $G_9 (k)$, we have $D_9 = \{3,6,7,10,11 \}$.  
Any core path in $G_r (k)$ of length at least three must contain 
at least one vertex from $D_r$.  


\section{Proof of Theorem \ref{theorem main}}

In this section we prove Theorem \ref{theorem main}.  Fix an odd integer $r \geq 5$ and 
let $k$ be an integer with $2k \geq r +5$. 

\subsection{Showing $G_r (k)$ is $C_r$-free}\label{showing cr free}

Any subgraph of $G_r(k)$ obtained by removing a vertex in $D_r$ is bipartite,
so any odd cycle in $G_r(k)$ must contain every vertex in $D_r$.  
A shortest cycle in $G_r(k)$ containing all vertices in $D_r$ has length $r+2$.  
This is easiest to see in the case that $r\equiv 1 ( \textup{mod}~3)$ because then, 
any pair of distinct vertices in $D_r$ are at distance at least 3 from each other,
and $|D_r| = \frac{r+2}{3}$.  
When $r \equiv 2 ( \textup{mod}~3)$, there is exactly one pair 
of adjacent vertices in $D_r$, but all other pairs are distance at least 3 from each other.
Similarly, when $r \equiv 0 ( \textup{mod}~3)$, there are
exactly two pairs of adjacent vertices in $D_r$,
but all other pairs are distance 3 from each other.  Nevertheless, since 
\[
|D_r| = 
\left\{
\begin{array}{ll}
\frac{r}{3} +2     &  \mbox{if $r \equiv 0 ( \textup{mod}~3)$} \\ 
\frac{r+1}{3} +1 &  \mbox{if $r \equiv 2 ( \textup{mod}~3),$}
\end{array}
\right.
\]
 a shortest cycle containing all vertices of $D_r$ must have 
length at least $r+2$.


\subsection{Showing $G_r(k)$ is $C_{2k}$-free}\label{showing c2k free}

First we consider even cycles that contain at most two vertices in $D_r$.  
The longest even cycle $C$ using only vertices 
in $\{i -1 \} \cup H_i (k) \cup \{ i+2 \}$ has length $2k-2$.
Indeed, such a cycle $C$ must contain at least one vertex in $A(k)$ and 
at least one vertex in $X(k)$.  This implies $C$ contains at most $|X(k)| - 1 $
vertices from $Y(k) \cup \{i +1 \}$, and 
at most $|A(k)| -1 $ vertices from $B(k) \cup \{i -1 \}$.  
Thus, 
\[
l(C) \leq 2 |A(k)| -1 + 2 |X(k)|  -1 = 2k-2.
\] 

Now consider 
a cycle $C$ 
that uses at least three vertices in $D_r$.
Such a cycle must then use all vertices in $D_r$, otherwise we could 
find a cut vertex on $C$.  
A cycle using all vertices in $D_r$ must have odd length.  This 
is because the subpath of $C$ contained in any $H_i (k)$ is a path between
some $a_i \in A_i (k)$ and $x_{i+1} \in X_{i+1 } (k)$ of odd length,
and all $a_i$,$x_{i+1}$-paths in $H_i (k)$ have odd length.  Therefore, we cannot change the 
parity of $l(C)$ by shortening or lengthening the subpath of $C$ contained in 
$H_i(k)$.


\subsection{Showing $C_{2k}$-saturation in $G_r(k)$}\label{showing saturation}

We must show that any two nonadjacent vertices are joined by a path of length $2k-1$.
We will start by examining paths that have at least one endpoint in some $A_i(k)$. When $k$ is even,
$G_r (k)$ has an automorphism that interchanges $A_i(k)$ with some $X_j(k)$.  When $k$ is odd,
there is no such automorphism because $A_i(k)$ has more vertices than $X_i(k)$.
Despite this, the arguments that work below for $A_i (k)$ also work for $X_i(k)$.
Furthermore, Lemma \ref{lemma paths} is symmetric with respect to $A(k)$ and $X(k)$,
and with respect to $B(k)$ and $Y(k)$.

\begin{center}
\underline{Paths with an endpoint in $A_i (k)$}
\end{center}

Let $a_i \in A_i (k)$.  Let $z$ be a vertex that is not adjacent to $a_i$.  We consider two cases.

\medskip
\noindent
\textbf{Case 1:} $z \in H_i(k)$

\smallskip

In this case, $z \in A_i(k)$ or $z\in Y_{i+1}(k)$.  
Either way, there is an $a_i$,$x_{i+1}$-core path $P_1$ whose first edge is $\{ a_i , i-1 \}$, 
whose last edge is $\{i+2 , x_{i+1} \}$ where $x_{i+1} \in X_{i+1}(k)$,
and $l(P_1) = r+1 $.
Since $G_r(k)$ contains at least two copies of $H(k)$,
there is an index $j \neq i$ such 
that $V(P_1) \cap V( H_j (k) ) = \{ a_j , x_{j+1} \}$ for some $a_j \in A_j (k)$,
$x_{j+1} \in X_{j+1} (k)$.  By Lemma \ref{lemma paths}, we can replace the
 edge $a_j x_{j+1}$ with an $a_j$,$x_{j+1}$-path $P_2$ of length $2k -r -2$
 where $V(P_2) \subseteq V( H_j (k) )$.  This gives an $a_i$,$x_{i+1}$-path $P_3$ of length
 \[
 l(P_3) = l(P_1) -1 + l(P_2) = r +1 -1 + 2k -r -2 = 2k -2.
 \]
 Lemma \ref{lemma paths} applies since $2k-r-2$ is odd, and 
 $1 \leq 2k -r -2 \leq 2k -3$ since we have assumed that $2k \geq r+5$.  
 We can then extend $P_3$ to a path of length $2k-1$ from $a_i$ to $a_i' \in A_i(k)$, or from
 $a_i$ to $y_{i+1} \in Y_{i+1}(k)$ by adding the edge $x_{i+1} a_i'$ or $x_{i+1} y_{i+1}$.

\medskip
\noindent
\textbf{Case 2:} $z \notin H_i(k)$

\smallskip

Since $z \notin H_i (k)$ and $z$ is not adjacent to $a_i$, we know that 
\[
z \notin \{ i-1 \} \cup A_i (k) \cup B_i (k) \cup X_{i+1}(k) \cup Y_{i+1}(k).
\]
We will consider two subcases.

\medskip
\noindent
\textbf{Subcase 2.1:} There is an $a_i$,$z$-path $P_1$ with $l(P_1) \in \{3,5, \dots , r \}$
and the first edge of $P_1$ is $a_i x_{i+1}$ for some $x_{i+1} \in X_{i+1} (k)$.

\smallskip

By Lemma \ref{lemma paths}, we may replace the edge $a_i x_{i+1}$ on $P_1$ with 
an $a_i ,x_{i+1}$-path $P_2$ with length $l(P_2) = 2k - l(P_1)$ 
were $V(P_2) \subseteq V(H_i (k))$.  This gives an $a_i$,$z$-path $P_3$ with length
\[
l(P_3) = l(P_1) - 1 + l(P_2) = 2k-1.
\]
Lemma \ref{lemma paths} applies since $l(P_1) \in \{3,5, \dots ,r \}$ so that 
$l(P_2) = 2k  - l(P_1)$ is odd, and $1 \leq 2k - l(P_1) \leq 2k-3$ as 
we have assumed $2k \geq r+5$.

\medskip
\noindent
\textbf{Subcase 2.2:} There is no $a_i$,$z$-path $P_1$ with $l(P_1) \in \{3,5, \dots , r \}$
where the first edge of $P_1$ is $a_i x_{i+1}$ for some $x_{i+1} \in X_{i+1} (k)$.

\smallskip
By Proposition \ref{proposition core paths}, 
there must be an $a_i$,$z$-path $P_2$ with $l(P_2) \in \{3,5, \dots ,r \}$ and 
where the first edge of $P_2$ is $\{a_i , i-1 \}$ (if $i=1$, then $i-1$ is replaced with $r+2$).
Observe $ V(P_2) \cap V ( H_i (k))  = \{ a_i \}$ 
since $z \notin H_i (k)$.  
We apply Lemma \ref{lemma paths} to find a $a_i$,$a_i'$-path $P_3$ where
$a_i' \in A_i (k)$, $V(P_3) \subseteq H_i (k)$, and $l(P_3) = 2k - 1 - l(P_2)$.  Note that since 
$l(P_2) \in \{3,5, \dots ,r \}$, we have $2k-1-r \leq l(P_3) \leq 2k-4$ where the first inequality follows
from the assumption $2k \geq r+5$.  If $P_4$ is the path obtained from $P_2$ by replacing $a_i$ with $a_i'$,
then the path $P_3 P_4$ is an $a_i$,$z$-path of length
\[
 l(P_3) + l(P_2) = 2k -1 - l(P_2) + l(P_2) = 2k-1.
\]

\smallskip

This completes the analysis in Case 2. 

 We conclude that for any vertex $a_i \in A_i (k)$ and 
any vertex $z$ not adjacent to $a_i$, there is a path of length $2k-1$ from $a_i$ to $z$.  
Because Lemma \ref{lemma paths} is symmetric with respect to $A(k)$ and $X(k)$,
the same arguments used above apply to $X_{i+1}(k)$.  Therefore, for any vertex 
$z_1 \in A_i (k) \cup X_{i+1} (k)$ and any vertex $z_2$ not adjacent to $z_1$,
there is a $z_1$,$z_2$-path of length $2k-1$.

\begin{center}
\underline{Paths with an endpoint in $B_i(k)$}
\end{center}

Let $b_i \in B_i (k)$ and let $z$ be a vertex not adjacent to $b_i$.  By what we have shown so 
far, we can assume that $z \notin A_j (k) \cup X_{j+1}(k)$ for every $j$.  Again, it will be convenient to
consider two cases.

\medskip
\noindent
\textbf{Case 1:} $z \in H_i(k)$

First suppose $z \in B_i (k)$, say $z = b_i'$.  Consider the 
path 
\[
P_1 = b_i,a_i,i-1,\dots,i+2,x_{i+1},a_i',b_i'
\]
where the subpath $a_i,i-1,\dots,i+2,x_{i+1}$
is a core path of length $r+1$, and $a_i' \in A_i (k) \backslash \{a_i \}$.  We have 
$l(P_1) = r+4$.  Since $G_r (k)$ contains two copies of $H(k)$, there is an index $j \neq i$
such that $V(P_1) \cap V( H_j (k)) = \{ a_j , x_{j+1} \}$ for some 
$a_j \in A_j (k)$, $x_{j+1} \in X_{j+1} (k)$.  By Lemma \ref{lemma paths}, we can replace the edge
$a_j x_{j+1}$ with an $a_j$,$x_{j+1}$-path of length $2k -r -4 $ to obtain
a $b_i$,$b_i'$-path $P_2$ of length $l(P_2) = l(P_1) - 1 + 2k -r -4 = 2k -1$.  
Lemma \ref{lemma paths} applies since $2k-r-4 $ is odd, and $1 \leq 2k -r -4 \leq 2k -3$ since 
we have assumed that $2k \geq r+5$.  

Now suppose $z=y_{i+1} \in Y_{i+1} (k)$. 
By Lemma \ref{lemma paths}, there is a path of length $2k-1$ from $b_i$ to $y_{i+1}$.

\medskip
\noindent
\textbf{Case 2:} $z \notin H_i(k)$

Again, Case 2 into divided into two subcases.

\medskip
\noindent
\textbf{Subcase 2.1:} There is a $b_i$,$z$-path $P_1$ with 
$l(P_1) \in \{3,5, \dots , r \}$, and the first two edges of $P_1$ are 
$b_i a_i$ and $a_i x_{i+1}$ for some $a_i \in A_i (k)$, $x_{i+1} \in X_{i+1} (k)$.  

\smallskip

By Lemma \ref{lemma paths}, we may replace the subpath $b_i a_i x_{i+1}$ on $P_1$ with 
a $b_i$,$x_{i+1}$-path $P_2$ of length $l(P_2) = 2k+1 - l(P_1)$ where 
$V(P_2) \subseteq V ( H_i (k))$.  Here we are using the fact that 
$z \notin H_i (k)$, and that $l(P_1) \in \{3, 5, \dots , r \}$ implies 
$l(P_2)$ is even and
\[
2k -r +1 \leq l(P_2) \leq 2k-2.
\]
This gives a $b_i$,$z$-path of length $2k-1$.

\medskip
\noindent
\textbf{Subcase 2.2:} There is no $b_i$,$z$-path $P_1$ with 
$l(P_1) \in \{3,5, \dots , r \}$ where the first two edges of $P_1$ are 
$b_i a_i$ and $a_i x_{i+1}$ for some $a_i \in A_i (k)$, $x_{i+1} \in X_{i+1} (k)$.  

\smallskip
By Proposition \ref{proposition core paths}, 
there must be a $b_i$,$z$-path $P_2$ with $l(P_2) \in \{3,5, \dots ,r \}$ and 
the first two edges of $P_2$ are $b_i a_i$ and $\{ a_i , i-1 \}$ where
$a_i \in A_i (k)$ (if $i=1$, then $i-1$ is replaced with $r+2$).
Since $ V(P_2) \cap V ( H_i (k))  = \{ b_i, a_i \}$,
we can replace the edge $b_i a_i$ with a $b_i$,$a_i$-path $P_3$ of length
$l(P_3) = 2k  - l (P_2)$ where $V(P_3) \subseteq V(H_i(k))$.  Note
that since $l(P_3) \in \{3,5, \dots ,r \}$, 
\[
2k -r \leq l(P_3) \leq 2k -3
\]
and $2k -r \geq 1$ by our assumption $2k \geq r+5$.  This gives 
a path of length $2k-1$ from $b_i$ to $z$.

\smallskip

This completes the analysis in Case 2.  We conclude that for any vertex $b_i \in B_i (k)$ and 
any vertex $z$ not adjacent to $b_i$, there is a path of length $2k-1$ from $b_i$ to $z$.  
Just like in the case of $A_i(k)$, the symmetry of Lemma \ref{lemma paths} with respect to $B(k)$ and 
$Y(k)$ implies that we have the same result for paths with an endpoint in $Y_i(k)$.

\begin{center}
\underline{Paths with both endpoints in $D_r$}
\end{center}

Let $i$ and $j$ be two vertices in $D_r$ that are not adjacent.  There is an
$i$,$j$-core path $P_1$ 
with $l(P_1) \in \{3,5, \dots , r \}$ and $V(P_1) \cap V(H_t (k) ) = \{a_t , x_{t+1} \}$
for some $t$.  By Lemma \ref{lemma paths}, we may replace the edge $a_t x_{t+1}$ 
with an $a_t$,$x_{t+1}$-path $P_2$ of length $l(P_2) = 2k - l(P_1)$ 
where 
$V(P_2) \subseteq V(H_t (k))$.
This gives an $i$,$j$-path of length $2k-1$.  Lemma \ref{lemma paths} applies since 
$2k - l(P_1)$ is odd and $2k -r \leq l(P_2) \leq 2k -3$.


\subsection{Finishing the proof of Theorem \ref{theorem main}}\label{finish proof}

We have shown that $G_r(k)$ is a $C_r$-free graph
that is $C_{2k}$-saturated.  The number of vertices in $G_r(k)$ is 
$\frac{r+2}{3} (2k+1)$ in the case that $r \equiv 1 (\textup{mod}~3)$,
$\frac{r+1}{3} (2k +1)$ in the case that $r \equiv 2 (\textup{mod}~3)$, and
$\frac{2rk}{3} + \frac{r+6}{3}$ when $r \equiv 0 ( \textup{mod}~3)$. 
In each case,
 $G_r (k)$ has at least $2kr$ vertices.
  
Let $t \geq 1$ and let $G_{r,t} (k)$ be the graph obtained by
adding $t$ new vertices $b_1', \dots ,b_t'$ to $G_r (k)$ 
where each $b_i'$ is adjacent to all vertices in $A_1(k)$, and 
these are the only vertices adjacent to $b_i'$.  
 Roughly speaking, we are duplicating vertices in $B_1(k)$   
(this idea
 was used in \cite{kmtt} to prove $\textup{sat}(n,C_3,C_{2k})=0$ for $k \geq 3$, $n \geq 2k+2$).  
 We complete the proof of Theorem \ref{theorem main}
 by showing that $G_{r,t}(k)$ is $C_r$-free and $C_{2k}$-saturated.
 
Let $B_1'(k) = B_1(k) \cup \{b_1' , \dots ,b_t' \}$.
Any cycle $C$ in $G_{r,t}(k)$ that contains $s$ vertices in $B_1'(k)$ must contain 
at least $s$ vertices in $A_1 (k)$.  Thus, at most $|A_1(k) | = |B_1(k)|$ vertices 
in $B_1'(k)$ can be used on $C$.  Since all vetices in the independent 
set $B_1'(k)$ have the same neighborhood, we can 
replace all vertices on $C$ that are in $B_1'(k)$ with vertices in $B_1(k)$ that are not on $C$.  
Therefore, adding 
$\{b_1', \dots ,b_t' \}$ to $G_r(k)$ to obtain $G_{r,t}(k)$ will not create cycles of new lengths.
Since $G_r(k)$ is $C_r$-free and $C_{2k}$-free, so is $G_{r,t}(k)$. 

We finish by showing $C_{2k}$-saturation in $G_{r,t}(k)$.   
Let $b_i' \in \{b_1' , \dots , b_t' \}$ and let $z$ be a vertex not adjacent to $b_i'$.  In $G_r(k)$,
there is an automorphism that interchanges any two given vertices in $B_1(k)$ so we can assume
that $z \neq b_1 \in B_1(k)$.  If $z \notin B_1'(k)$, then in $G_r(k)$, there 
is a $b_1$,$z$-path of length $2k-1$.  Replacing $b_1$ on this path with $b_i'$ gives 
a path of length $2k-1$ from $b_i'$ to $z$.  If $z\in B_1'(k)$, then we can choose 
a path of length $2k-1$ in $G_r(k)$ from $b_1 \in B_1(k)$ to $b_2 \in B_1(k)$.  
Replacing $b_1$ and $b_2 $ on this path with $b_i'$ and $z$, respectively, gives a path of length $2k-1$ from 
$b_i'$ to $z$.  
 
 
 \section{Acknowledgements}
 
 The author would like to thank J\"urgen Kritschgau, Abhishek Methuku, and Michael Tait
 for many helpful discussions.  


 \end{document}